\pdfoutput=1
\documentclass[11pt]{amsart}

\usepackage{amssymb}
\usepackage[pdfauthor={Federico Galetto},
pdftitle={On the ideal generated by all squarefree monomials of a given
  degree},colorlinks=true]{hyperref}
\usepackage{enumitem}
\usepackage{ytableau}
\usepackage{blkarray}
\usepackage{tikz}
\usetikzlibrary{matrix,arrows}
 
\theoremstyle{definition}
\newtheorem{theorem}{Theorem}[section]
\newtheorem{proposition}[theorem]{Proposition}
\newtheorem{corollary}[theorem]{Corollary}
\newtheorem{lemma}[theorem]{Lemma}
\newtheorem{definition}[theorem]{Definition}
\newtheorem{remark}[theorem]{Remark}
\newtheorem{example}[theorem]{Example}
\newtheorem*{thm*}{Theorem}

\DeclareMathOperator{\rank}{rank}
\DeclareMathOperator{\sgn}{sgn}
\DeclareMathOperator{\SYT}{SYT}
\DeclareMathOperator{\im}{im}

\renewcommand{\geq}{\geqslant}
\renewcommand{\leq}{\leqslant}
\newcommand{\NN}{\mathbb{N}}

\title[]{On the ideal generated by all squarefree monomials of a given
  degree}
\date{\today}
\author{Federico Galetto}
\address{McMaster University\\1280 Main St W\\Hamilton Hall 407\\
  Hamilton, ON, L8S 4K1\\Canada}
\email{galettof@math.mcmaster.ca}
\urladdr{http://math.galetto.org}

\begin{document}

\begin{abstract}
  An explicit construction is given of a minimal free resolution of
  the ideal generated by all squarefree monomials of a given
  degree. The construction relies upon and exhibits the natural action
  of the symmetric group on the syzygy modules. The resolution is
  obtained over an arbitrary coefficient ring; in particular, it is
  characteristic free. Two applications are given: an equivariant
  resolution of De Concini-Procesi rings indexed by hook partitions,
  and a resolution of FI-modules.
\end{abstract}

\maketitle{}

\tableofcontents{}

\section{Introduction}
\label{sec:introduction}

Let $A$ be an associative, commutative ring with unit.  Let $R_n$
denote $A [x_1,\ldots,x_n]$, the polynomial ring in $n$ variables with
coefficients in $A$.

\begin{definition}
  Let $I_{d,n}$ be the ideal of $R_n$ generated by all squarefree
  monomials of degree $d$ in $x_1,\ldots,x_n$.
\end{definition}

The symmetric group $\mathfrak{S}_n$ acts on $R_n$ by permuting the
variables. This action is $A$-linear, and is compatible with grading
and multiplication in $R_n$. Observe that the ideal $I_{d,n}$ is
stable under the action of $\mathfrak{S}_n$. Therefore
$\mathfrak{S}_n$ acts on the minimal free resolutions of
$I_{d,n}$. The main goal of this paper is to construct an
$\mathfrak{S}_n$-equivariant minimal free resolution of $I_{d,n}$ that
describes this action explicitly. Our main theorem (Theorem
\ref{thm:2}) can be stated as follows.

\begin{thm*}
  Let $n$, $d$, and $i$ be integers. For $1\leq d\leq n$ and
  $0\leq i\leq n-d$, we define the $A[\mathfrak{S}_n]$-module
  \begin{equation*}
    U^{d,n}_i :=
    \operatorname{Ind}^{\mathfrak{S}_n}_{\mathfrak{S}_{d+i} \times
      \mathfrak{S}_{n-d-i}} ( S^{(d,1^i)} \otimes_A S^{(n-d-i)} ),
  \end{equation*}
  For all other values of $n$, $d$, and $i$, we set $U^{d,n}_i:=0$.
  The ideal $I_{d,n}$ admits an $\mathfrak{S}_n$-equivariant
  minimal free resolution $F^{d,n}_\bullet$ with
  \begin{equation*}
    F^{d,n}_i := U^{d,n}_i \otimes_A R_n (-d-i).
  \end{equation*}
\end{thm*}

Our construction has some desirable features:
\begin{itemize}[wide]
\item the differentials are described explicitly (see Definition
  \ref{def:1});
\item the resolution is realized over an arbitrary coefficient ring,
  in particular it is characteristic free;
\item the representation theory leads to an easy combinatorial
  interpretation of Betti numbers in terms of Young tableaux
  (Corollary \ref{cor:1}).
\end{itemize}

The ideals $I_{d,n}$ appear in the study of De Concini-Procesi
rings. Introduced in \cite{MR629470}, De Concini-Procesi rings are
quotients of a polynomial ring that give a presentation for the
cohomology ring of Springer fibers (in type A). These rings also
appear in the study of nilpotent orbits \cite{MR646820}. In general,
the Betti numbers and minimal free resolutions of De Concini-Procesi
are not known. In the case of De Concini-Procesi rings indexed by hook
partitions, the Betti numbers were described in \cite{MR2371263}. The
proof, which holds over a field, recognizes $I_{d,n}$ as a monomial
ideal with linear quotients to compute its Betti numbers, and then
uses an iterated mapping cone to obtain the Betti numbers of a certain
De Concini-Procesi ring. Using the same mapping cone procedure, we
give an application of our main result and describe the modules in an
$\mathfrak{S}_n$-equivariant minimal free resolution of a De
Concini-Procesi ring indexed by a hook partition (see Theorem
\ref{thm:3}).  We believe that an approach using equivariant
resolutions could facilitate the task of finding Betti numbers for
other De Concini-Procesi rings as well.

The resolutions $F^{d,n}_\bullet$ of our main theorem have another
interesting property. Namely, for a fixed $d$, there are natural
$\mathfrak{S}_n$-equivariant maps of complexes
$F^{d,n}_\bullet \to F^{d,n+1}_\bullet$. These maps allow us to
assemble the complexes $F^{d,n}_\bullet$ to obtain a resolution of
FI-modules in the sense of \cite{MR3357185} (see Theorem \ref{thm:4}).

We also note that $I_{d,n}$ is the defining ideal of a star
configuration of $n$ hyperplanes in projective $(n-1)$-space (see
\cite{MR3003727}). The Betti numbers of these ideals were computed
(over a field) in \cite[Corollary 3.5]{MR3299722}. Despite being a
special case of star configurations, it was shown in \cite{1507.00380}
that the ideals $I_{d,n}$ actually govern the theory of arbitrary star
configurations and their symbolic powers. For this reason we believe
it would be particularly interesting to extend the main result of this
paper to describe $\mathfrak{S}_n$-equivariant minimal free
resolutions of the symbolic powers of $I_{d,n}$.

I am grateful to J.~Weyman for introducing me to De Concini-Procesi
rings, which ultimately lead to this project. I am particularly
indebted to A.~Hoefel and D.~Wehlau for many useful discussions in the
early stages of this work. Special thanks go to H.~Abe for providing a
reference that shortened the proof of the main theorem.  While working
on this project, I was partially supported by an NSERC grant.

\section{Betti numbers}
\label{sec:betti-numbers}

As explained in the introduction, the Betti numbers of $I_{d,n}$ were
previously computed, over a field, by several authors using various
techniques. In addition, the Betti numbers of $I_{d,n}$ can be
recovered, over various coefficient rings, using Stanley-Reisner
theory \cite[\S 5]{MR1251956}, properties of squarefree stable ideals
\cite{MR1630500}, or the generic perfection of the Eagon-Northcott
complex \cite[Theorem 3.5]{MR953963}. Since we will be using these
Betti numbers later, we show a different way to compute them over an
arbitrary coefficient ring using mapping cones.

\begin{theorem}
  \label{thm:1}
  Let $n$ and $d$ be integers with $1\leq d\leq n$.
  The Betti numbers of $I_{d,n}$ are
  \begin{equation*}
    \beta_{i,j} (I_{d,n}) =
    \begin{cases}
      \binom{n}{d+i} \binom{d+i-1}{i},
      &\text{if }0 \leq i \leq n-d \text{ and } j = d+i,\\
      0, &\text{otherwise}.
    \end{cases}
  \end{equation*}
\end{theorem}

\begin{proof}
  The proof is by double induction on $n$ and $d$.

  For $n=d=1$, we have $I_{1,1} = (x_1)$, for which the result
  clearly holds.

  For $n>1$, $d=1$, the ideal $I_{1,n} = (x_1,\ldots,x_n)$ is
  minimally resolved by a Koszul complex, whose Betti numbers match
  those in the statement.

  Consider the case $n>1$, $d>1$. There is a short exact sequence
  \begin{equation}
    \label{eq:1}
    0 \to I_{d,n-1} \hookrightarrow I_{d-1,n-1}
    \to I_{d-1,n-1} / I_{d,n-1} \to 0
  \end{equation}
  of $R_{n-1}$-modules, the second arrow being an inclusion.  Let
  $F^{d,n-1}_\bullet$ and $F^{d-1,n-1}_\bullet$ be minimal free
  resolutions of the left and middle modules in \eqref{eq:1}.  The
  inclusion $I_{d,n-1} \hookrightarrow I_{d-1,n-1}$ extends to a map
  of complexes $F^{d,n-1}_\bullet \to F^{d-1,n-1}_\bullet$, whose
  mapping cone $C_\bullet$ is a free resolution of
  $I_{d-1,n-1} / I_{d,n-1}$. By the inductive hypothesis,
  $C_i = F^{d,n-1}_{i-1} \oplus F^{d-1,n-1}_i$ is generated in degree
  $d+i-1$. It follows that $C_\bullet$ is linear and hence minimal.

  The squarefree monomials generating $I_{d,n}$ are of two kinds: the
  ones that are not divisible by $x_n$, and the ones that are. The
  former coincide with the generators of $I_{d,n-1}$, the latter are
  the generators of $I_{d-1,n-1}$ multiplied by $x_n$.
  This leads to the short exact sequence
  \begin{multline}
    \label{eq:2}
    0 \to (I_{d-1,n-1} / I_{d,n-1}) \otimes_{R_{n-1}} R_n (-1)
    \xrightarrow{\cdot x_n}\\ (R_{n-1} / I_{d,n-1}) \otimes_{R_{n-1}} R_n
    \to R_n / I_{d,n} \to 0
  \end{multline}
  of $R_n$-modules, where the second arrow is multiplication by $x_n$.

  Since $R_n$ is a free $R_{n-1}$-module, it is flat. Thus
  $F^{d,n-1}_\bullet \otimes_{R_{n-1}} R_n$ is an exact complex. In
  fact, if we augment it with the obvious map
  $F^{d,n-1}_0 \otimes_{R_{n-1}} R_n \to R_n$, we obtain a minimal
  free resolution of the middle term in \eqref{eq:2}. Similarly,
  $C_\bullet \otimes_{R_{n-1}} R_n (-1)$ is a minimal free resolution
  of the left term. The second map of sequence \eqref{eq:2} extends to
  a map of complexes between the two resolutions just described. The
  mapping cone $D_\bullet$ of this map is a free resolution of
  $R_n / I_{d,n}$. Note $D_0 = R_n$ and, for $i>0$, we have
  \begin{equation}
    \label{eq:3}
    D_i = (C_{i-1} \otimes_{R_{n-1}} R_n (-1))
    \oplus (F^{d,n-1}_{i-1} \otimes_{R_{n-1}} R_n),
  \end{equation}
  so $D_i$ is generated in degree $d+i-1$. Therefore $D_\bullet$ is
  minimal.

  If $F^{d,n}_\bullet$ is a minimal free resolution of $I_{d,n}$, then
  $F^{d,n}_i \cong D_{i+1}$ is generated in degree $d+i$. Moreover,
  equation \eqref{eq:3} gives
  \begin{equation}
    \label{eq:6}
    F^{d,n}_i \cong ((F^{d,n-1}_{i-1} \oplus F^{d-1,n-1}_i)
    \otimes_{R_{n-1}} R_n (-1))
    \oplus (F^{d,n-1}_i \otimes_{R_{n-1}} R_n).
  \end{equation}
  We deduce that
  \begin{equation*}
    \begin{split}
      \beta_{i,d+i} (I_{d,n}) =
      {}&\tbinom{n-1}{d+i-1}\tbinom{d+i-2}{i-1}
      +\tbinom{n-1}{d+i-1}\tbinom{d+i-2}{i}
      +\tbinom{n-1}{d+i}\tbinom{d+i-1}{i}=\\
      ={}&\tbinom{n-1}{d+i-1}\tbinom{d+i-1}{i}
      +\tbinom{n-1}{d+i}\tbinom{d+i-1}{i}=\\
      ={}&\tbinom{n}{d+i}\tbinom{d+i-1}{i}
    \end{split}
  \end{equation*}
  as desired. Clearly all other Betti numbers are zero.
\end{proof}

\begin{example}[$n=4,d=2$]
  The nonzero Betti numbers of $I_{2,4}$ are
  \begin{equation*}
    \beta_{0,2} (I_{2,4}) = 6, \quad
    \beta_{1,3} (I_{2,4}) = 8, \quad
    \beta_{2,4} (I_{2,4}) = 3.
  \end{equation*}
\end{example}

\section{Representations}
\label{sec:representations}

Specht modules are representations of the symmetric group.  Their
construction and properties are discussed in detail in
\cite{MR513828}.  We briefly recall a presentation for Specht modules
following \cite[\S 7.4]{MR1464693}.

Let $A [\mathfrak{S}_n]$ be the group algebra of the symmetric group
$\mathfrak{S}_n$ over $A$.  Let $\lambda$ be a partition of $n$.  A
(Young) tableau of shape $\lambda$ is a filling of the Young diagram
associated to $\lambda$ (we use the English notation for diagrams,
cf.~\cite[p.~2]{MR1354144}). The Specht module $S^\lambda$ is the
$A [\mathfrak{S}_n]$-module generated by the equivalence classes $[T]$
of tableaux of shape $\lambda$ with entries in $[n]=\{1,\ldots,n\}$
modulo the following relations.
\begin{description}[wide]
\item[Alternating columns] $\sigma [T] = \sgn (\sigma) [T]$ for all
  $\sigma \in \mathfrak{S}_n$ preserving the columns of $T$
  ($\sgn (\sigma)$ denotes the sign of $\sigma$).
\item[Shuffling relations] $[T]=\sum [T']$, where the sum is over all $T'$
  obtained from $T$ by exchanging the top $k$ elements of the
  $(j+1)$-st column of $T$ with $k$ elements in the $j$-th column of
  $T$, preserving the vertical order of each set of $k$ elements.
\end{description}

We are primarily interested in hook partitions, i.e., partitions of
the form $(d,1^i)$. We illustrate the relations above with an example
using a hook partition.

\begin{example}
  Consider the hook partition $(3,1,1)$ of $5$.  Since columns are
  alternating in $S^{(3,1,1)}$, we have
  \ytableausetup{smalltableaux,centertableaux}
  \begin{equation*}
    \Bigg[\ytableaushort{215,3,4}\Bigg] = -\Bigg[\ytableaushort{315,2,4}\Bigg] =
    \Bigg[\ytableaushort{315,4,2}\Bigg].
  \end{equation*}
  Using the shuffling relation involving the first two columns, we get
  \begin{equation*}
    \Bigg[\ytableaushort{215,3,4}\Bigg] = \Bigg[\ytableaushort{125,3,4}\Bigg] +
    \Bigg[\ytableaushort{235,1,4}\Bigg] + \Bigg[\ytableaushort{245,3,1}\Bigg],
  \end{equation*}
  whereas the one involving the second and third columns gives
  \begin{equation*}
    \Bigg[\ytableaushort{215,3,4}\Bigg] = \Bigg[\ytableaushort{251,3,4}\Bigg].
  \end{equation*}
\end{example}

Let us recall some facts from the theory of Specht modules.
\begin{itemize}[wide]
\item For all $\lambda$, $S^\lambda$ is a free $A$-module.  Its rank
  can be computed with the hook length formula (cf.~\cite[Theorem
  20.1]{MR513828}). In particular, in the case of a hook, we have
  \begin{equation}
    \label{eq:4}
    \rank ( S^{(d,1^i)} ) =
    \tfrac{(d+i)!}{(d+i)(d-1)!i!} = \tbinom{d+i-1}{i}.
  \end{equation}
\item A standard tableau is one whose rows are strictly increasing
  from left to right, and whose columns are strictly increasing from
  top to bottom. If $\lambda$ is a partition of $n$, then the
  equivalence classes of standard tableaux of shape $\lambda$ with
  entries in $[n]$ form an $A$-basis of $S^\lambda$.
\end{itemize}
We will denote by $\SYT (\lambda,[n])$ the set of standard tableaux of
shape $\lambda$ with entries in $[n]$. Note that $\SYT (\lambda,[n])$
is defined even when $\lambda$ is a partition of some non negative
integer $m\neq n$.

\begin{definition}
  Let $n$, $d$, and $i$ be integers. For $1\leq d\leq n$ and
  $0\leq i\leq n-d$, we define an $A[\mathfrak{S}_n]$-module by
  setting
  \begin{equation*}
    U^{d,n}_i :=
    \operatorname{Ind}^{\mathfrak{S}_n}_{\mathfrak{S}_{d+i} \times
      \mathfrak{S}_{n-d-i}} ( S^{(d,1^i)} \otimes_A S^{(n-d-i)} ),
  \end{equation*}
  where the right hand side of the assignment is the
  $A[\mathfrak{S}_n]$-module induced by the
  $A[\mathfrak{S}_{d+i} \times \mathfrak{S}_{n-d-i}]$-module
  $S^{(d,1^i)} \otimes_A S^{(n-d-i)}$.  The group
  $\mathfrak{S}_{d+i} \times \mathfrak{S}_{n-d-i}$ is realized as the
  subgroup of permutations of $\mathfrak{S}_n$ that preserve the
  subsets $\{1,\ldots,d+i\}$ and $\{d+i+1,\ldots,n\}$. If any of the
  inequalities involving $n$, $d$, and $i$ are violated, we set
  $U^{d,n}_i:=0$.
\end{definition}
The modules $U^{d,n}_i$ will play a central role in the rest of this
paper.  We record some properties of the modules $U^{d,n}_i$.

\begin{proposition}
  \label{pro:1}
  Let $n$, $d$, and $i$ be integers with $1\leq d\leq n$, and
  $0\leq i\leq n-d$.
  \begin{enumerate}[label=(\alph*),wide]
  \item\label{item:1} The module $U^{d,n}_i$ is a free $A$-module and
    \begin{equation*}
      \rank (U^{d,n}_i) = \tbinom{n}{d+i}\tbinom{d+i-1}{i}.
    \end{equation*}
  \item\label{item:2} The module $U^{d,n}_i$ is isomorphic to the
    $A[\mathfrak{S}_n]$-module generated by the equivalence classes of
    tableaux of shape $(d,1^i)$ with entries in $[n]$ modulo
    alternating columns and shuffling relations. The equivalence
    classes of standard tableaux form an $A$-basis of $U^{d,n}_i$.
  \item\label{item:3} The module $U^{d,n}_i$ is a principal
    $A[\mathfrak{S}_n]$-module generated by the equivalence class of
    any tableau of shape $(d,1^i)$ with entries in $[n]$.
  \end{enumerate}
\end{proposition}

\begin{proof}
  \begin{enumerate}[label=(\alph*),wide]
  \item By definition of induced module, we have
    \begin{equation}\label{eq:5}
      U^{d,n}_i \cong \bigoplus
      \sigma (S^{(d,1^i)} \otimes_A S^{(n-d-i)}),
    \end{equation}
    where the direct sum is over a set of representatives for the
    cosets of $\mathfrak{S}_{d+i}\times \mathfrak{S}_{n-d-i}$ in
    $\mathfrak{S}_n$.  Each summand is a free $A$-module, therefore
    so is $U^{d,n}_i$.  Note that $\rank (S^{(n-d-i)}) = 1$, hence
    \begin{equation*}
      \rank (U^{d,n}_i)
      = \tfrac{|\mathfrak{S}_n|}{|\mathfrak{S}_{d+i}
        \times \mathfrak{S}_{n-d-i}|} \tbinom{d+i-1}{i}
      =\tbinom{n}{d+i}\tbinom{d+i-1}{i}.
    \end{equation*}
  \item Consider the module $S^{(d,1^i)} \otimes_A S^{(n-d-i)}$. The
    factor $S^{(n-d-i)}$ has rank one, hence it does not play a role
    in what follows. The factor $S^{(d,1^i)}$ is generated by the
    equivalence classes of tableaux of shape $(d,1^i)$ with entries in
    $[d+i]$ modulo alternating columns and shuffling relations.
    Moreover, $S^{(d,1^i)}$ has a basis given by the equivalence
    classes of standard tableaux.

    For each coset of $\mathfrak{S}_{d+i} \times \mathfrak{S}_{n-d-i}$
    in $\mathfrak{S}_n$ there is a unique representative $\sigma$
    which is an increasing function on $\{1,\ldots,d+i\}$ and on
    $\{d+i+1,\ldots,n\}$. Denote by $C$ the collection of all coset
    representatives with this property.
    
    Given a representative $\sigma \in C$ and a tableau $T$ of shape
    $(d,1^i)$ with entries in $[d+i]$, we can apply $\sigma$ to all
    entries of $T$ to obtain a tableau $\sigma T$ of shape $(d,1^i)$
    with entries in $\{\sigma(1),\ldots,\sigma(d+i)\}$.  Moreover, if
    $T$ is standard, then $\sigma T$ is again standard because
    $\sigma$ is increasing on $\{1,\ldots,d+i\}$.

    Therefore we can identify the direct summand
    $\sigma (S^{(d,1^i)} \otimes_A S^{(n-d-i)})$ in \eqref{eq:5} with
    the $A$-module generated by equivalence classes of tableaux of
    shape $(d,1^i)$ with entries in $\{\sigma(1),\ldots,\sigma(d+i)\}$
    modulo alternating columns and shuffling relations. In particular,
    the equivalence classes of standard tableaux form an $A$-basis.
    As $\sigma$ runs over $C$, we get the desired result.
  \item Fix a tableau $T$ of shape $(d,1^i)$ with entries in
    $[n]$. Given any other tableau $T'$ of shape $(d,1^i)$ with
    entries in $[n]$, there is a permutation
    $\sigma \in \mathfrak{S}_n$ such that $T'=\sigma T$. Thus we have
    $[T']=\sigma [T]$ in $U^{d,n}_i$. Now the statement follows from
    part \ref{item:2}.
  \end{enumerate}
\end{proof}

\begin{remark}
  If $A$ is a field, then the Littlewood-Richardson rule (see \cite[\S
  16]{MR513828}) implies that the module $U^{d,n}_i$ admits a
  filtration whose associated graded object is $\bigoplus S^\lambda$,
  where the direct sum is over all partitions
  $\lambda=(\lambda_1,\lambda_2,\ldots)$ of $n$ such that
  $\lambda_1 \geq d$, $\lambda_2 \leq d$, and $\lambda_i = 1$ for all
  $i$ with $3\leq i \leq n+2-\lambda_1-\lambda_2$. In particular, note
  that this is a multiplicity-free decomposition in terms of Specht
  modules. The same description can also be obtained using the simpler
  Pieri rule (cf.~\cite[\S 2.2]{MR1464693}). If $A$ is a field of
  characteristic zero, then the same rules give a decomposition of
  $U^{d,n}_i$ as a direct sum of simple $A[\mathfrak{S}_n]$-modules.
\end{remark}

\begin{remark}
  The restriction of $U^{d,n}_i$ to $\mathfrak{S}_{n-1}$ has the
  following interesting property:
  \begin{equation*}
    \operatorname{Res}^{\mathfrak{S}_n}_{\mathfrak{S}_{n-1}}
    ( U^{d,n}_i ) \cong
    U^{d,n-1}_{i-1} \oplus U^{d-1,n-1}_i \oplus U^{d,n-1}_i.
  \end{equation*}
  This is easy to prove using Proposition \ref{pro:1} \ref{item:2}.
  Notice also that this property matches the isomorphism in equation
  \eqref{eq:6}.
\end{remark}

\section{Equivariant resolution}
\label{sec:equiv-struct-1}

In the section, we describe an $\mathfrak{S}_n$-equivariant minimal
free resolution of $I_{d,n}$. The free modules and the differentials
are constructed using the representations $U^{d,n}_i$ introduced in \S
\ref{sec:representations}. As the definitions are very explicit, this
allows for a direct proof of exactness, which is the content of our
main theorem. We conclude this section with different interpretations
of the Betti numbers of $I_{d,n}$.

\subsection{Modules}
\label{sec:modules}

\begin{definition}
  Let $n$, $d$, and $i$ be integers, with $n>0$. We define an
  $R_n [\mathfrak{S}_n]$-module by setting
  \begin{equation*}
    F^{d,n}_i := U^{d,n}_i \otimes_A R_n (-d-i).
  \end{equation*}
\end{definition}

We write $p[T]$ for the simple tensor
$[T]\otimes p \in F^{d,n}_i$, where $p\in R_n$ and $T$ is a tableau of
shape $(d,1^i)$ with entries in $[n]$.

\begin{proposition}
  \label{pro:4}
  Let $n$, $d$, and $i$ be integers, with $1\leq d\leq n$ and
  $0\leq i\leq n-d$.
  \begin{enumerate}[label=(\alph*),wide]
  \item The module $F^{d,n}_i$ is a free $R_n$-module and
    \begin{equation*}
      \rank (F^{d,n}_i) = \tbinom{n}{d+i}\tbinom{d+i-1}{i}.
    \end{equation*}
  \item The module $F^{d,n}_i$ is isomorphic to the
    $R_n [\mathfrak{S}_n]$-module generated by the equivalence classes
    of tableaux of shape $(d,1^i)$ with entries in $[n]$ modulo
    alternating columns and shuffling relations. The equivalence
    classes of standard tableaux form an $R_n$-basis of $F^{d,n}_i$.
  \item The module $F^{d,n}_i$ is a principal
    $R_n [\mathfrak{S}_n]$-module generated by the equivalence class of
    any tableau of shape $(d,1^i)$ with entries in $[n]$.
  \end{enumerate}
\end{proposition}

\begin{proof}
  Everything follows from the corresponding statements in Proposition
  \ref{pro:1}.
\end{proof}

\subsection{Differentials}
\label{sec:differentials}

Consider the tableau
\begin{equation*}
  \ytableausetup{boxsize=2em}
  T =
  \begin{ytableau}
    a_0 & b_1 & \ldots & b_{d-1}\\
    a_1\\
    \vdots\\
    a_i
  \end{ytableau}
\end{equation*}
of shape $(d,1^i)$ with entries in $[n]$. For $0\leq j\leq i$, define
$T\setminus a_j$ as the tableau of shape $(d,1^{i-1})$ with entries in
$[n]$ obtained from $T$ by removing the box containing $a_j$ and
sliding upward the boxes below it. For clarity, we have
\begin{equation*}
  \ytableausetup{boxsize=2em}
  T\setminus a_j =
  \begin{ytableau}
    a_0 & b_1 & \ldots & b_{d-1}\\
    a_1\\
    \vdots\\
    a_{j-1}\\
    a_{j+1}\\
    \vdots\\
    a_i
  \end{ytableau}
  \qquad
  T\setminus a_0 =
  \begin{ytableau}
    a_1 & b_1 & \ldots & b_{d-1}\\
    a_2\\
    \vdots\\
    a_i
  \end{ytableau}
\end{equation*}
respectively when $1\leq j\leq i$ and when $j=0$.

\begin{definition}
  \label{def:1}
  Let $n$, $d$, and $i$ be integers, with $1\leq d\leq n$ and
  $1\leq i\leq n-d$.  Fix a tableau $T$ of shape $(d,1^i)$ with
  entries $a_0,\ldots,a_i,b_1,\ldots,b_{d-1}$ in $[n]$ as above.
  Since $F^{d,n}_i$ is a principal $R_n [\mathfrak{S}_n]$-module
  generated by $[T]$, we define a map
  $\partial^{d,n}_i \colon F^{d,n}_i \to F^{d,n}_{i-1}$ of
  $R_n [\mathfrak{S}_n]$-modules by setting
  \begin{equation*}
    \partial^{d,n}_i ([T]) := \sum_{j=0}^i (-1)^{i-j} x_{a_j} [T\setminus a_j]
  \end{equation*}
  and extending by $R_n [\mathfrak{S}_n]$-linearity.
\end{definition}
Note that the same formula for $\partial^{d,n}_i$ holds for any other
tableau $T'$ of shape $(d,1^i)$ with entries in $[n]$.

\begin{example}[$n=4,d=2,i=2$]
  We have, for example,
  \begin{equation*}
    \ytableausetup{smalltableaux,centertableaux}
    \partial^{2,4}_2 \left(\Bigg[\ytableaushort{12,3,4}\Bigg]\right)=
    x_1 \Big[\ytableaushort{32,4}\Big]
    -x_3 \Big[\ytableaushort{12,4}\Big]
    +x_4 \Big[\ytableaushort{12,3}\Big].
  \end{equation*}
  Expressing this in terms of standard tableaux, we get
  \begin{equation*}
    \begin{split}
    \ytableausetup{smalltableaux,centertableaux}
    \partial^{2,4}_2 \left(\Bigg[\ytableaushort{12,3,4}\Bigg]\right)
    &=x_1 \left(\Big[\ytableaushort{23,4}\Big]
      + \Big[\ytableaushort{34,2}\Big]\right)
    -x_3 \Big[\ytableaushort{12,4}\Big]
    +x_4 \Big[\ytableaushort{12,3}\Big] =\\
    &=x_1 \Big[\ytableaushort{23,4}\Big]
    -x_1 \Big[\ytableaushort{24,3}\Big]
    -x_3 \Big[\ytableaushort{12,4}\Big]
    +x_4 \Big[\ytableaushort{12,3}\Big].
  \end{split}
  \end{equation*}
  In fact, the matrix of $\partial^{2,4}_2$ in the bases of standard tableaux is
  \begin{equation*}
    \renewcommand*{\arraystretch}{1.8}
    \begin{blockarray}{cccc}
      \begin{block}{c[ccc]}
        \Big[\ytableaushort{23,4}\Big] & x_1 & x_1 & 0\\
        \Big[\ytableaushort{13,4}\Big] & 0 & -x_2 & 0\\
        \Big[\ytableaushort{12,4}\Big] & -x_3 & 0 & 0\\
        \Big[\ytableaushort{24,3}\Big] & -x_1 & 0 & x_1\\
        \Big[\ytableaushort{14,3}\Big] & 0 & 0 & -x_2\\
        \Big[\ytableaushort{14,2}\Big] & 0 & 0 & x_3\\
        \Big[\ytableaushort{12,3}\Big] & x_4 & 0 & 0\\
        \Big[\ytableaushort{13,2}\Big] & 0 & x_4 & 0\\
      \end{block}
      & \Bigg[\ytableaushort{12,3,4}\Bigg]
      & \Bigg[\ytableaushort{13,2,4}\Bigg]
      & \Bigg[\ytableaushort{14,2,3}\Bigg]\\
    \end{blockarray}
  \end{equation*}
\end{example}

\begin{proposition}
  Let $n$ and $d$ be integers with $1\leq d\leq n$.  The sequence of
  maps $\partial^{d,n}_i$ forms a minimal complex of free
  $R_n$-modules.
\end{proposition}
\begin{proof}
  The computation showing that
  $\partial^{d,n}_i \partial^{d,n}_{i+1} =0$ is essentially the same
  as the one for differentials in a Koszul complex. Minimality is
  obvious from the definition of $\partial^{d,n}_i$.
\end{proof}

\begin{definition}
  Let $n$ and $d$ be integers with $1\leq d\leq n$. We denote by
  $F^{d,n}_\bullet$ the complex defined by the sequence of maps
  $\partial^{d,n}_i$.
\end{definition}

\begin{example}[$n=4,d=2$]
  \label{exa:1}
  As $R_4$-modules, we have $F^{2,4}_0 \cong R_4 (-2)^6$,
  $F^{2,4}_1 \cong R_4 (-3)^8$, and $F^{2,4}_2 \cong R_4 (-4)^3$.  The
  complex $F^{2,4}_\bullet$ looks as follows.
  \begin{equation*}
    F^{2,4}_0
    \xleftarrow{
      \left[
      \begin{smallmatrix}
        -x_2 & -x_1 & 0 & -x_2 & -x_1 & 0 & 0 & 0\\
        0 & 0 & -x_1 & x_3 & 0 & -x_1 & 0 & 0\\
        0 & 0 & 0 & 0 & x_3 & x_2 & 0 & 0\\
        x_4 & 0 & 0 & 0 & 0 & 0 & -x_1 & -x_1\\
        0 & x_4 & 0 & 0 & 0 & 0 & 0 & x_2\\
        0 & 0 & x_4 & 0 & 0 & 0 & x_3 & 0
      \end{smallmatrix}
      \right]
    }
    F^{2,4}_1
    \xleftarrow{
      \left[
        \begin{smallmatrix}
          x_1 & x_1 & 0\\
          0 & -x_2 & 0\\
          -x_3 & 0 & 0\\
          -x_1 & 0 & x_1\\
          0 & 0 & -x_2\\
          0 & 0 & x_3\\
          x_4 & 0 & 0\\
          0 & x_4 & 0
        \end{smallmatrix}
      \right]
    }
    F^{2,4}_2
  \end{equation*}
\end{example}

\begin{remark}
  \label{rem:1}
  By Theorem \ref{thm:1}, the complex $F^{d,n}_\bullet$ has the same
  Betti numbers as a minimal free resolution of $I_{d,n}$.
\end{remark}

\subsection{Exactness}
\label{sec:exactness}

\begin{lemma}
  \label{lem:1}
  Let $\varphi^{d,n}_i$ denote
  $(\partial^{d,n}_i)_{d+i} \colon (F^{d,n}_i)_{d+i} \to
  (F^{d,n}_{i-1})_{d+i}$, i.e., the restriction of $\partial^{d,n}_i$ to
  degree $d+i$. Then $\varphi^{d,n}_i$ admits a left inverse.
\end{lemma}

Before proving the lemma, we illustrate $\varphi^{d,n}_i$ with an
example.

\begin{example}[$n=3,d=2,i=1$]
  We write an explicit matrix for $\varphi^{2,3}_1$.  Note that the
  domain of $\varphi^{2,3}_1$ is isomorphic to $U^{2,3}_1$.  On the
  other hand, the codomain of $\varphi^{2,3}_1$ is isomorphic to
  $U^{2,3}_0 \otimes_A (R_3)_1$.  For the domain, we choose the
  $A$-basis given by $[T]$, where $T \in \SYT ((2,1),[3])$. For the
  codomain, we choose the $A$-basis given by elements $x_j [T']$,
  where $x_j$ is a variable in $R_n$ and $T' \in \SYT ((2),[3])$.
  \begin{equation*}
    \renewcommand*{\arraystretch}{1.5}
    \begin{blockarray}{ccc}
      \begin{block}{c[cc]}
        x_1 [\ytableaushort{23}] & -1 & -1\\{}
        x_2 [\ytableaushort{23}] & 0 & 0\\{}
        x_3 [\ytableaushort{23}] & 0 & 0\\{}
        x_1 [\ytableaushort{13}] & 0 & 0\\{}
        x_2 [\ytableaushort{13}] & 0 & 1\\{}
        x_3 [\ytableaushort{13}] & 0 & 0\\{}
        x_1 [\ytableaushort{12}] & 0 & 0\\
        x_2 [\ytableaushort{12}] & 0 & 0\\
        x_3 [\ytableaushort{12}] & 1 & 0\\
      \end{block}
      & \Big[\ytableaushort{12,3}\Big]
      & \Big[\ytableaushort{13,2}\Big]\\
    \end{blockarray}
  \end{equation*}
  Note that the submatrix on the fifth and ninth row from the top is a
  permutation matrix, hence invertible over $A$. Therefore
  $\varphi^{2,3}_1$ admits a left inverse.
\end{example}

\begin{proof}
  Observe that $(F^{d,n}_i)_{d+i} \cong U^{d,n}_i$. By Proposition
  \ref{pro:1}, $U^{d,n}_i$ is a free $A$-module with a basis
  $\mathcal{B}$ consisting of equivalence classes $[T]$ for
  $T\in \SYT ((d,1^i),[n])$.  Similarly,
  $(F^{d,n}_{i-1})_{d+i} \cong U^{d,n}_{i-1} \otimes_A (R_n)_1$ is a
  free $A$-module with a basis $\mathcal{C}$ consisting of elements
  $x_j [T']$, where $x_j$ is a variable in $R_n$ and
  $T'\in \SYT((d,1^{i-1}),[n])$.  Let $M$ be the matrix of
  $\varphi^{d,n}_i$ relative to the bases $\mathcal{B}$ and
  $\mathcal{C}$. We will show that $M$ admits a left inverse by
  selecting some rows giving a submatrix of $M$ which is invertible
  over $A$.

  Define the set
  \begin{equation*}
    \mathcal{C}' :=
    \{x_{a_i} [T\setminus a_i] \in \mathcal{C} \mid T \in \SYT ((d,1^i),[n])\}.
  \end{equation*}
  The elements $\mathcal{C}'$ are obtained from tableaux on $(d,1^i)$
  by removing the bottom box of the first column and multiplying by
  the variable with index the entry of the removed box.  The
  cardinalities of $\mathcal{B}$, $\mathcal{C}$, and $\mathcal{C}'$
  are related by the formula
  $|\mathcal{B}|=|\mathcal{C'}| < |\mathcal{C}|$. In particular, this
  shows that $M$ always has more rows than columns.  Let $N$ be the
  square submatrix of $M$ on the rows corresponding to elements of
  $\mathcal{C}'$.

  Given $[T] \in \mathcal{B}$, we have
  \begin{equation*}
    \partial^{d,n}_i ([T]) = \sum_{j=0}^i (-1)^{i-j}
    x_{a_j} [T\setminus a_j].
  \end{equation*}
  For $1\leq j\leq i$, the elements $x_{a_j} [T\setminus a_j]$ are in
  $\mathcal{C}$. The only one belonging to $\mathcal{C}'$ is
  $x_{a_i} [T\setminus a_i]$, which appears with coefficient 1. If
  $T\setminus a_0$ is standard, then $x_{a_0} [T\setminus a_0]$ is in
  $\mathcal{C}$ but not in $\mathcal{C}'$.  If $T\setminus a_0$ is not
  standard, then $x_{a_0} [T\setminus a_0]$ can be expanded into an
  $A$-linear combination of basis elements $x_{a_0} [T']$ for certain
  tableaux $T' \in \SYT((d,1^{i-1}),[n])$. Note that the entries of
  any such $T'$ form a subset of
  $\{a_1,\ldots,a_i,b_1,\ldots,b_{d-1}\}$; in particular, all entries
  of $T'$ are bigger than $a_0$. This implies that $x_{a_0} [T']$ is
  not in $\mathcal{C}'$.  We deduce that the column of $N$
  corresponding to $[T]$ has a single nonzero entry, and this entry is
  is equal to 1. Therefore $N$ is a permutation matrix, hence it is
  invertible over $A$.
\end{proof}

\begin{theorem}\label{thm:2}
  Let $n$ and $d$ be integers with $1\leq d\leq n$.  The complex
  $F^{d,n}_\bullet$ is an $\mathfrak{S}_n$-equivariant minimal free
  resolution of the $R_n$-module $I_{d,n}$.
\end{theorem}
\begin{proof}
  The module $F^{d,n}_0$ has an $R_n$-basis $[T]$ with
  $T\in \SYT ((d),[n])$.  Define a map of $R_n$-modules
  $\partial^{d,n}_0 \colon F^{d,n}_0 \to I_{d,n}$ by sending the
  equivalence class of the tableau
  \begin{equation*}
    \ytableausetup{boxsize=1.5em}
    \begin{ytableau}
      b_1 & \ldots & b_d\\
    \end{ytableau}
  \end{equation*}
  to $x_{b_1} \cdots x_{b_d} \in I_{d,n}$. The map
  $\partial^{d,n}_0$ is clearly
  $\mathfrak{S}_n$-equivariant. Moreover, $\partial^{d,n}_0$ is
  surjective because $I_{d,n}$ is generated by the squarefree
  monomials of degree $d$ in $R_n$.

  Next we show that $\partial^{d,n}_0 \partial^{d,n}_1 = 0$.
  Consider the tableau
  \begin{equation*}
    T = 
    \ytableausetup{boxsize=2em}
    \begin{ytableau}
      a_0 & b_1 & \ldots & b_{d-1}\\
      a_1\\
    \end{ytableau}
  \end{equation*}
  of shape $(d,1)$ with entries in $[n]$. We have
  \begin{equation*}
    \begin{split}
      \partial^{d,n}_0 \partial^{d,n}_1 ([T]) &= \partial^{d,n}_0 (
      -x_{a_0}
      [T\setminus a_0] + x_{a_1} [T\setminus a_1]) =\\
      &=-x_{a_0} x_{a_1} x_{b_1} \cdots x_{b_{d-1}} +x_{a_1} x_{a_0}
      x_{b_1} \cdots x_{b_{d-1}} = 0.
      \end{split}
    \end{equation*}
  Therefore $F^{d,n}_\bullet$ can be extended, via the map
  $\partial^{d,n}_0$, to an $\mathfrak{S}_n$-equivariant complex of
  $R_n$-modules $0 \leftarrow I_{d,n} \leftarrow F^{d,n}_\bullet$,
  which is exact at $I_{d,n}$. If we can show this complex is exact
  everywhere else, then the theorem will follow.  We will prove the
  complex is exact at $F^{d,n}_i$ proceeding by induction on $i$.

  For the base case, let $i=0$. We need to show
  $\ker (\partial^{d,n}_0) = \im (\partial^{d,n}_1)$.  Consider
  $\varphi^{d,n}_1$, the restriction of $\partial^{d,n}_1$ to degree
  $d+1$.  By Lemma \ref{lem:1}, $\varphi^{d,n}_1$ admits a left
  inverse, which is a map of $A$-modules from $(F^{d,n}_0)_{d+1}$ to
  $(F^{d,n}_1)_{d+1}$.  Denote by
  $\psi^{d,n}_1 \colon (\ker (\partial^{d,n}_0))_{d+1} \to
  (F^{d,n}_1)_{d+1}$ the $A$-module map obtained by restricting the
  left inverse of $\varphi^{d,n}_1$ to the degree $d+1$ component of
  $\ker (\partial^{d,n}_0)$ (see the diagram below).
  \begin{center}
      \begin{tikzpicture}
        \matrix (m) [matrix of math nodes, row sep=3em, column
        sep=2em, text height=2ex, text depth=0.25ex]
        { (F^{d,n}_0)_{d+1} & (F^{d,n}_1)_{d+1} \\
          (\ker (\partial^{d,n}_0))_{d+1}& \\ };
        \path[<-,font=\scriptsize] (m-1-1) edge node[auto]
        {$\varphi^{d,n}_1$} (m-1-2) (m-1-2) edge[bend left=15]
        node[auto] {$\psi^{d,n}_1$} (m-2-1);
        \path[right hook->,font=\scriptsize] (m-2-1) edge node[auto]
        {} (m-1-1);
      \end{tikzpicture}    
  \end{center}
  Since
  $\im (\partial^{d,n}_1) \subseteq \ker (\partial^{d,n}_0)$, we have
  $\im (\varphi^{d,n}_1) \subseteq (\ker
  (\partial^{d,n}_0))_{d+1}$. Then we can consider the composition
  $\psi^{d,n}_1 \varphi^{d,n}_1$, which is, by construction, the
  identity map of $(F^{d,n}_1)_{d+1}$.  It follows that $\psi^{d,n}_1$
  is surjective.  Since $\partial^{d,n}_0$ surjects onto $I_{d,n}$,
  Remark \ref{rem:1} implies that $\ker (\partial^{d,n}_0)$ is
  generated in degree $d+1$ and that $(\ker (\partial^{d,n}_0))_{d+1}$
  is a free $A$-module of rank $\binom{n}{d+1} \binom{d}{1}$. At the
  same time, $(F^{d,n}_1)_{d+1}$ is also free of rank
  $\binom{n}{d+1} \binom{d}{1}$. Since $\psi^{d,n}_1$ is a surjection
  between free modules of the same rank, it is an isomorphism (see
  \cite[Chapter 3, Exercise 15]{MR0242802}).  Now $\varphi^{d,n}_1$ is
  the right inverse of $\psi^{d,n}_1$, hence it gives an isomorphism
  of $A$-modules between $(F^{d,n}_1)_{d+1}$ and
  $(\ker (\partial^{d,n}_0))_{d+1}$. We deduce that
  $(\ker (\partial^{d,n}_0))_{d+1} = \im (\varphi^{d,n}_1) = (\im
  (\partial^{d,n}_1))_{d+1}$. Given that $\ker (\partial^{d,n}_0)$ is
  generated in degree $d+1$, we conclude that
  $\ker (\partial^{d,n}_0) = \im (\partial^{d,n}_1)$ as desired.

  For the inductive step, let $i>0$. The proof is essentially the same
  as for the base step. We need to show
  $\ker (\partial^{d,n}_i) = \im (\partial^{d,n}_{i+1})$.  The map
  $\varphi^{d,n}_{i+1}$, the restriction of $\partial^{d,n}_{i+1}$ to
  degree $d+i+1$, admits a left inverse by Lemma \ref{lem:1}.  Denote
  by
  $\psi^{d,n}_{i+1} \colon (\ker (\partial^{d,n}_i))_{d+i+1} \to
  (F^{d,n}_{i+1})_{d+i+1}$ the $A$-module map obtained by restricting
  the left inverse of $\varphi^{d,n}_{i+1}$ to the degree $d+i+1$
  component of $\ker (\partial^{d,n}_i)$ (see the diagram below).
  \begin{center}
      \begin{tikzpicture}
        \matrix (m) [matrix of math nodes, row sep=3em, column
        sep=2em, text height=2ex, text depth=0.25ex]
        { (F^{d,n}_i)_{d+i+1} & (F^{d,n}_{i+1})_{d+i+1} \\
          (\ker (\partial^{d,n}_i))_{d+i+1}& \\ };
        \path[<-,font=\scriptsize] (m-1-1) edge node[auto]
        {$\varphi^{d,n}_{i+1}$} (m-1-2) (m-1-2) edge[bend left=15]
        node[auto] {$\psi^{d,n}_{i+1}$} (m-2-1);
        \path[right hook->,font=\scriptsize] (m-2-1) edge node[auto]
        {} (m-1-1);
      \end{tikzpicture}    
  \end{center}
  Since
  $\im (\varphi^{d,n}_{i+1}) \subseteq (\ker
  (\partial^{d,n}_i))_{d+i+1}$, the composition
  $\psi^{d,n}_{i+1} \varphi^{d,n}_{i+1}$ is the identity map of
  $(F^{d,n}_{i+1})_{d+i+1}$.  It follows that $\psi^{d,n}_{i+1}$ is
  surjective. By induction,
  $\im (\partial^{d,n}_i) = \ker (\partial^{d,n}_{i-1})$. Hence Remark
  \ref{rem:1} implies that $\ker (\partial^{d,n}_i)$ is generated in
  degree $d+i+1$ and that $(\ker (\partial^{d,n}_i))_{d+i+1}$ is a
  free $A$-module of rank $\binom{n}{d+i+1} \binom{d+i}{i+1}$.  Note
  that $(F^{d,n}_{i+1})_{d+i+1}$ is also free of rank
  $\binom{n}{d+i+1} \binom{d+i}{i+1}$. Being a surjection between free
  modules of the same rank, $\psi^{d,n}_{i+1}$ is an isomorphism.
  Thus $\varphi^{d,n}_{i+1}$ gives an isomorphism between
  $(F^{d,n}_{i+1})_{d+i+1}$ and $(\ker
  (\partial^{d,n}_i))_{d+i+1}$. It follows that
  $(\ker (\partial^{d,n}_i))_{d+i+1} = \im (\varphi^{d,n}_{i+1}) =
  (\im (\partial^{d,n}_{i+1}))_{d+i+1}$. Given that
  $\ker (\partial^{d,n}_i)$ is generated in degree $d+i+1$, we
  conclude that $\ker (\partial^{d,n}_i) = \im (\partial^{d,n}_{i+1})$
  as desired.
\end{proof}

The following corollary gives a representation theoretic description
of the syzygy modules, and a combinatorial interpretation of the Betti
numbers of $I_{d,n}$. The proof follows directly from Theorem
\ref{thm:2}.
\begin{corollary}
  \label{cor:1}
  Let $n$ and $d$ be integers with $1\leq d\leq n$.
  \begin{enumerate}[wide,label=(\alph*)]
  \item There are isomorphisms of $A[\mathfrak{S}_n]$-modules
    \begin{equation*}
      \operatorname{Tor}^{R_n}_i (I_{d,n},A)_j \cong
      \begin{cases}
        U^{d,n}_i,
        &\text{if }j = d+i,\\
        0, &\text{otherwise}.
      \end{cases}
    \end{equation*}
  \item The Betti numbers of $I_{d,n}$ are given by the formula
    \begin{equation*}
      \beta_{i,j} (I_{d,n}) =
      \begin{cases}
        |\SYT ((d,1^i),[n])|,
        &\text{if }j = d+i,\\
        0, &\text{otherwise}.
      \end{cases}
    \end{equation*}
  \end{enumerate}
\end{corollary}

\section{Applications}
\label{sec:applications}

\subsection{De Concini-Procesi rings}
\label{sec:de-concini-procesi}

Let $n$ be a fixed positive integer. For each partition $\mu$ of $n$,
there is a De Concini-Procesi ideal
$I_\mu \subset R_n = A[x_1,\ldots,x_n]$, and a corresponding De
Concini-Procesi ring $R_n /I_\mu$. Historically, these ideals and
quotients have been studied in the case when the coefficient ring $A$
is a field, although they can be defined more generally. The original
definition is in \cite{MR629470}, along with an explicit set of
generators. Over the years, other generating sets have been described
that are simpler and/or smaller than the original one; see for example
\cite{MR2448086,MR685425}. Since we are only interested in the case of
hook partitions, we describe the De Concini-Procesi ideal in this case
following \cite[Proposition 3.4]{MR2371263}.

\begin{proposition}
  \label{pro:2}
  Let $\mu = (n-d+1,1^{d-1})$, with $1\leq d\leq n$. Then we have
  \begin{equation*}
    I_\mu = (e_1,\ldots,e_{d-1}) + I_{d,n},
  \end{equation*}
  where $e_i$ is the $i$-th symmetric polynomial.
\end{proposition}

In Proposition \ref{pro:2}, we recognize our monomial ideal
$I_{d,n}$. The elementary symmetric polynomials making up the rest of
the ideal have another useful property first presented in
\cite[Proposition 4.3]{MR2371263}.
\begin{proposition}
  \label{pro:3}
  Let $n$ and $d$ be integers with $1\leq d\leq n$. The residue
  classes of $e_1,\ldots,e_{d-1}$ in $R_n /I_{d,n}$ form a regular
  sequence.
\end{proposition}

Using this property, R.~Biagioli, S.~Faridi, and M.~Rosas observed
that a free resolution of $R_n / I_\mu$ can be obtained from one of
$R_n / I_{d,n}$ via iterated mapping cones. Moreover, Proposition
\ref{pro:3}, guarantees that the mapping cones are minimal. The
details of the procedure are discussed in \cite{MR2371263}.

This mapping cone construction extends immediately to the equivariant
case. In fact, we have an $\mathfrak{S}_n$-equivariant resolution of
$I_{d,n}$ from Theorem \ref{thm:2} and the mapping cones are induced
from multiplication by an elementary symmetric polynomial, which is an
$\mathfrak{S}_n$-invariant. We illustrate with an example.

\begin{example}[$n=4,d=2$]
  \label{exa:2}
  Based on Proposition \ref{pro:2}, we have
  $I_{(3,1)} = (e_1) + I_{2,4}$. From Proposition \ref{pro:3}, we know
  that $e_1$ is regular on $R_4 / I_{2,4}$. Therefore we have a short
  exact sequence
  \begin{equation*}
    0 \to R_4 / I_{2,4} \xrightarrow{\cdot e_1} R_4 / I_{2,4}
    \to R_4 / I_{(3,1)} \to 0
  \end{equation*}
  where the second map is multiplication by $e_1$.

  From Example \ref{exa:1}, we have the following resolution of
  $R_4 / I_{2,4}$.
  \begin{equation*}
    R_4 \xleftarrow{\partial^{4,2}_0}
    F^{2,4}_0 \xleftarrow{\partial^{4,2}_1}
    F^{2,4}_1 \xleftarrow{\partial^{4,2}_2}
    F^{2,4}_2 \leftarrow 0
  \end{equation*}
  The map $\partial^{4,2}_0$ was defined in the proof of Theorem
  \ref{thm:2}.  Multiplication by $e_1$ in the short exact
  sequence above extends to a map between two copies of the resolution
  of $R_4 / I_{2,4}$. The mapping cone of this map of complexes looks
  as follows.
  \begin{multline*}
    R_4 \xleftarrow{
      \left[
      \begin{smallmatrix}
        \partial^{2,4}_0 & e_1
      \end{smallmatrix}
      \right]
    } F^{2,4}_0 \oplus R_4(-1) \xleftarrow{
      \left[
      \begin{smallmatrix}
        \partial^{2,4}_1 & e_1\\
        0 & -\partial^{2,4}_0
      \end{smallmatrix}
      \right]
    } F^{2,4}_1 \oplus F^{2,4}_0 (-1) \leftarrow\\
    \xleftarrow{
      \left[
      \begin{smallmatrix}
        \partial^{2,4}_2 & e_1\\
        0 & -\partial^{2,4}_1
      \end{smallmatrix}
      \right]
    } F^{2,4}_2 \oplus F^{2,4}_1 (-1) \xleftarrow{
      \left[
      \begin{smallmatrix}
        e_1\\
        -\partial^{2,4}_2
      \end{smallmatrix}
      \right]
    } F^{2,4}_2 (-1) \leftarrow 0
  \end{multline*}
  By the general theory, this complex is an
  $\mathfrak{S}_n$-equivariant minimal free resolution of $R_4 / I_{(3,1)}$.
\end{example}

A complete description of the differentials in the iterated mapping
cone complex resolving $I_\mu$ would be notationally
cumbersome. Instead we focus on a description of the
modules. Following the example of Biagioli, Faridi, and Rosas, we will
use a bigraded Poincar\'e series. We start by recalling some
terminology.

The Grothendieck ring of $A[\mathfrak{S}_n]$ is the free abelian group
generated by isomorphism classes $[P]$ of finitely generated
projective $A[\mathfrak{S}_n]$-modules modulo the relations
$[P']-[P]+[P'']=0$ for every short exact sequence
\begin{equation*}
  0\to P' \to P \to P''\to 0.
\end{equation*}
Addition is defined by $[P_1] + [P_2]:=[P_1 \oplus P_2]$, while
multiplication is defined by $[P_1] [P_2]:=[P_1 \otimes P_2]$. Note
that the class $[S^{(n)}]$ of the rank one trivial
$A[\mathfrak{S}_n]$-module is the identity element for multiplication;
thus we will simply write 1 for $[S^{(n)}]$.

\begin{remark}
  The $A[\mathfrak{S}_n]$-module $U^{d,n}_i$ is projective. To see
  this, consider the following diagram of $A[\mathfrak{S}_n]$-modules
  with exact rows.
  \begin{center}
    \begin{tikzpicture}
      \matrix (m) [matrix of math nodes, row sep=3em, column sep=2.5em,
      text height=2ex, text depth=0.25ex]
      { & U^{d,n}_i & \\
        M & N & 0 \\ };
      \path[->,font=\scriptsize]
      (m-2-1) edge node[auto] {$\psi$} (m-2-2)
      (m-2-2) edge node[auto] {} (m-2-3)
      (m-1-2) edge node[auto] {$\varphi$} (m-2-2);
      \path[<-,dashed,font=\scriptsize]
      (m-2-1) edge node[auto] {$\rho$} (m-1-2);
    \end{tikzpicture}
  \end{center}
  Let $T$ be a tableau of shape $(d,1^i)$ with entries in $[n]$, so
  that $[T]$ generates $U^{d,n}_i$. Since $\psi$ is surjective, there
  exists $m\in M$ such that $\psi (m) = \varphi ([T])$. The
  $A$-submodule of $U^{d,n}_i$ generated by $[T]$ is free. Therefore
  we can define $\rho \colon U^{d,n}_i \to M$ by setting
  $\rho ([T]) := m$, and extending by
  $A[\mathfrak{S}_n]$-linearity. This guarantees that $\rho$ is a map
  of $A[\mathfrak{S}_n]$-modules. Moreover, the equality
  $\rho ([T]) = m$, implies that $\varphi = \psi\rho$.

  Since $U^{d,n}_i$ is a finitely generate projective
  $A[\mathfrak{S}_n]$-module, we can consider its class
  $[U^{d,n}_i]$ in the Grothendieck ring of $A[\mathfrak{S}_n]$.
\end{remark}

\begin{definition}
  Let $M$ be a finitely generated graded
  $R_n [\mathfrak{S}_n]$-module.  Suppose that, for all integers
  $i,j$, $\operatorname{Tor}^{R_n}_i (M,A)_j$ is a finitely generated
  projective $A[\mathfrak{S}_n]$-module. The
  \emph{$\mathfrak{S}_n$-equivariant bigraded Poincar\'e series} of
  $M$ is the power series
  \begin{equation*}
    \operatorname{P}^{\mathfrak{S}_n}_M (q,t) := \sum_{i,j}
    [\operatorname{Tor}^{R_n}_i (M,A)_j] q^i t^j
  \end{equation*}
  in the variables $q,t$ with coefficients in the Grothendieck ring of
  $A[\mathfrak{S}_n]$.
\end{definition}

This equivariant Poincar\'e series is simply a compact device to keep
track of the representations appearing in an equivariant resolution.

\begin{example}
  Comparing with Example \ref{exa:2}, we have the equality
  \begin{equation*}
    \operatorname{P}^{\mathfrak{S}_n}_{R_4/I_{2,4}} (q,t)=
    1+[U^{2,4}_0]qt^2+[U^{2,4}_1]q^2t^3+[U^{2,4}_2]q^3t^4,
  \end{equation*}
  and also
  \begin{equation*}
    \begin{split}
      \operatorname{P}^{\mathfrak{S}_n}_{R_4/I_{(3,1)}} (q,t)
      &=1+[U^{2,4}_0]qt^2+qt+([U^{2,4}_1]+[U^{2,4}_0])q^2t^3+\\
      &\quad +([U^{2,4}_2]+[U^{2,4}_1])q^3t^4+[U^{2,4}_2]q^4t^5=\\
    &=(1+qt)(1+[U^{2,4}_0]qt^2+[U^{2,4}_1]q^2t^3+[U^{2,4}_2]q^3t^4).
    \end{split}
  \end{equation*}
\end{example}

Finally, we are in a position to state the main result for De
Concini-Procesi rings indexed by hook partitions. This theorem is an
immediate consequence of Theorem \ref{thm:2} and the discussion of
this section.

\begin{theorem}
  \label{thm:3}
  Let $\mu = (n-d+1,1^{d-1})$, with $1\leq d\leq n$. Then we have
  \begin{equation*}
    \operatorname{P}^{\mathfrak{S}_n}_{R_n / I_\mu} (q,t)=
    \prod_{k=1}^{d-1} (1+qt^k) \left(1+
      \sum_{i=0}^{n-d} [U^{d,n}_i] q^{i+1} t^{d+i} \right).
  \end{equation*}
\end{theorem}

\subsection{A resolution of FI-modules}
\label{sec:resol-fi-modul}

The framework of FI-modules, developed by T.~Church, J.S.~Ellenberg,
and B.~Farb, allows us to assemble all the complexes
$F^{d,n}_\bullet$, for a fixed $d$, into a single comprehensive
structure. For all details on FI-modules, we refer the reader to
\cite{MR3357185}. Our base ring will be $A$, and it should replace any
instance of $k$ occurring in the reference mentioned. We do not appeal
to any deep result about FI-modules, as we are merely interested in
the language they provide.  Explicitly, we will use the notions of
FI-module (Definition 1.1), morphism of FI-modules (Definition 2.1.1),
graded FI-module (Remark 2.1.5), and exactness in the category of
FI-modules (which follows from Remark 2.1.2).

\begin{remark}
  The notation in \cite{MR3357185} uses a capital letter, say $V$, for
  an FI-module, and $V_n$ for the object that $V$ associates to the
  set $[n]$. To distinguish FI-modules from other entities, we will
  use a capital calligraphic letter, say $\mathcal{V}$, for an
  FI-module. Moreover, our FI-modules will carry subscripts for
  homological dimension and graded components. To avoid using another
  subscript, we will write $\mathcal{V} (n)$ for the object that $V$
  associates to the set $[n]$.
\end{remark}

\begin{definition}
  For a positive integer $d$, define $\mathcal{I}_d$ to be the graded
  FI-module that associates:
  \begin{itemize}[wide]
  \item to the set $[n]$ the graded $A$-module $I_{d,n}$;
  \item to an injection $\varepsilon \colon [n] \to [m]$ the map
    $\mathcal{I}_d (\varepsilon) \colon I_{d,n} \to I_{d,m}$ of graded
    $A$-modules defined by
    $p (x_1,\ldots,x_n) \mapsto p
    (x_{\varepsilon(1)},\ldots,x_{\varepsilon(n)})$.
  \end{itemize}
\end{definition}

Let $T$ be a tableau of shape $(d,1^i)$ with entries in $[n]$. If
$\varepsilon \colon [n] \to [m]$ is an injection, then denote by
$\varepsilon (T)$ the tableau of shape $(d,1^i)$ with entries in $[m]$
obtained by replacing each entry $i$ of $T$ by $\varepsilon (i)$.

\begin{definition}
  For integers $d$ and $i$, with $d>0$ and $i\geq 0$, define
  $\mathcal{F}^d_i$ to be the FI-module that associates:
  \begin{itemize}[wide]
  \item to the set $[n]$ the graded $A$-module $F^{d,n}_i$;
  \item to an injection $\varepsilon \colon [n] \to
    [m]$ the map $\mathcal{F}^d_i (\varepsilon) \colon F^{d,n}_i \to
    F^{d,m}_i$ of graded $A$-modules given by $p (x_1,\ldots,x_n) [T]
    \mapsto p
    (x_{\varepsilon(1)},\ldots,x_{\varepsilon(n)})[\varepsilon
    (T)]$, where $T$ is a tableau of shape
    $(d,1^i)$ with entries in $[n]$.
  \end{itemize}
\end{definition}

\begin{proposition}
  \label{pro:5}
  Let $d$ and $i$ be positive integers. For all integers
  $m,n$, with $m\geq n\geq 0$, and all injections $\varepsilon\colon
  [n]\to [m]$, the diagrams of graded $A$-modules and maps
  \begin{center}
    \begin{tikzpicture}
      \matrix (m) [matrix of math nodes, row sep=3em, column sep=2.5em,
      text height=2ex, text depth=0.25ex]
      { I_{d,m} & F^{d,m}_0 \\
        I_{d,n} & F^{d,n}_0 \\};
      \path[<-,font=\scriptsize]
      (m-1-1) edge node[left] {$\mathcal{I}_d (\varepsilon)$} (m-2-1)
      (m-2-1) edge node[auto] {$\partial^{d,n}_0$} (m-2-2)
      (m-1-1) edge node[auto] {$\partial^{d,m}_0$} (m-1-2)
      (m-1-2) edge node[auto] {$\mathcal{F}^d_0 (\varepsilon)$} (m-2-2);
    \end{tikzpicture}
    \begin{tikzpicture}
      \matrix (m) [matrix of math nodes, row sep=3em, column sep=2.5em,
      text height=2ex, text depth=0.25ex]
      { F^{d,m}_{i-1} & F^{d,m}_i \\
        F^{d,n}_{i-1} & F^{d,n}_i \\};
      \path[<-,font=\scriptsize]
      (m-1-1) edge node[left] {$\mathcal{F}^d_{i-1} (\varepsilon)$} (m-2-1)
      (m-2-1) edge node[auto] {$\partial^{d,n}_i$} (m-2-2)
      (m-1-1) edge node[auto] {$\partial^{d,m}_i$} (m-1-2)
      (m-1-2) edge node[auto] {$\mathcal{F}^d_i (\varepsilon)$} (m-2-2);
    \end{tikzpicture}
  \end{center}
  are commutative.
\end{proposition}
\begin{proof}
  The proof of commutativity of the two diagrams is similar. We write
  a proof for the diagram on the right.

  If $F^{d,n}_i =0$, then the diagram is obviously commutative. Hence
  we assume that $F^{d,n}_i \neq 0$. It follows that all other modules
  appearing in the diagram are also nonzero.

  The group $\mathfrak{S}_m$ acts on $F^{d,m}_{i-1}$ and $F^{d,m}_i$,
  while $\mathfrak{S}_n$ acts on $F^{d,n}_{i-1}$ and $F^{d,n}_i$.  By
  construction, the map $\partial^{d,m}_i$ is
  $\mathfrak{S}_m$-equivariant and $\partial^{d,n}_i$ is
  $\mathfrak{S}_n$-equivariant.  Fix an injection
  $\varepsilon \colon [n]\to [m]$. The choice of $\varepsilon$ gives
  rise to a natural identification of $\mathfrak{S}_n$ with a subgroup
  of $\mathfrak{S}_m$. With this convention, all maps in the diagram
  are $\mathfrak{S}_n$-equivariant.
  
  Let $T$ be a tableau of shape $(d,1^i)$ with entries in $[n]$. By
  Proposition \ref{pro:4}, $F^{d,n}_i$ is generated, as an
  $R_n [\mathfrak{S}_n]$-module, by $[T]$. Therefore it is enough to
  prove commutativity of the diagram holds for $[T]$. On one hand we
  have
  \begin{equation*}
    \begin{split}
      \mathcal{F}^d_{i-1} (\varepsilon) \partial^{d,n}_i ([T])
      &= \mathcal{F}^d_{i-1} (\varepsilon) \left(\sum_{j=0}^i
        (-1)^{i-j} x_{a_j} [T \setminus a_j] \right) =\\
      &= \sum_{j=0}^i (-1)^{i-j} x_{\varepsilon (a_j)}
      [\varepsilon (T \setminus a_j)] =\\
      &= \sum_{j=0}^i (-1)^{i-j} x_{\varepsilon (a_j)} [\varepsilon
      (T) \setminus \varepsilon (a_j)].
    \end{split}
  \end{equation*}
  On the other hand
  \begin{equation*}
    \partial^{d,m}_i \mathcal{F}^d_i (\varepsilon) ([T]) =
    \partial^{d,m}_i ([\varepsilon (T)]) =
    \sum_{j=0}^i (-1)^{i-j} x_{\varepsilon (a_j)}
    [\varepsilon (T) \setminus \varepsilon (a_j)].
  \end{equation*}
  This shows that
  $\mathcal{F}^d_{i-1} (\varepsilon) \partial^{d,n}_i ([T])
  = \partial^{d,m}_i \mathcal{F}^d_i (\varepsilon) ([T])$. Therefore
  the diagram commutes.
\end{proof}

As a consequence of Proposition \ref{pro:5}, we are allowed to make
the following definitions.

\begin{definition}
  For a positive integer $d$, let
  $\partial^d_0 \colon \mathcal{F}^d_0 \to \mathcal{I}_d$ be the
  morphism of graded FI-modules defined by setting
  $\partial^d_0 ([n]) := \partial^{d,n}_0$ for all integers $n$.
\end{definition}

\begin{definition}
  For positive integers $d$ and $i$, let
  $\partial^d_i \colon \mathcal{F}^d_i \to \mathcal{F}^d_{i-1}$ be the
  morphism of graded FI-modules defined by setting
  $\partial^d_i ([n]) := \partial^{d,n}_i$ for all integers $n$.
\end{definition}

\begin{theorem}
  \label{thm:4}
  For each positive integer $d$, there is a resolution of graded
  FI-modules
  \begin{equation*}
    0 \leftarrow \mathcal{I}_d \xleftarrow{\partial^d_0} \mathcal{F}^d_0
    \xleftarrow{\partial^d_1} \mathcal{F}^d_1
    \xleftarrow{\partial^d_2} \mathcal{F}^d_2
    \leftarrow \ldots \leftarrow \mathcal{F}^d_{i-1}
    \xleftarrow{\partial^d_i} \mathcal{F}^d_i \leftarrow \ldots
  \end{equation*}
\end{theorem}
\begin{proof}
  The sequence of graded FI-modules and morphisms in the statement of
  the theorem is functorial on the category of finite sets with
  injections.  Thus, given a nonnegative integer $n$, the sequence
  applies to the set $[n]$ to produce a complex of $R_n$-modules,
  namely the complex $F^{d,n}_\bullet$.  By definition, the sequence
  in the statement of the theorem is exact if and only if
  $F^{d,n}_\bullet$ is exact for all $n\in\NN$. The latter is true by
  virtue of Theorem \ref{thm:2}.
\end{proof}

\begin{remark}
  All graded FI-modules appearing in this section are of finite type
  in the sense of \cite[Definition 4.2.1]{MR3357185}.
\end{remark}

\begin{remark}
  One can write an analogue of Theorem \ref{thm:4} for De
  Concini-Procesi rings indexed by partitions $\mu = (n-d+1,1^{d-1})$.
\end{remark}

\end{document}